\documentclass[11pt,a4paper,oneside]{amsart}
\usepackage[utf8]{inputenc}
\usepackage[T1]{fontenc}
\newtheorem{thm}{Theorem}

\newtheorem{cor}[thm]{Corollary}

\theoremstyle{definition}

\newtheorem{ex}[thm]{Example}

\def\R{\mathbb{R}}
\def\N{\mathbb{N}}

\author[A. Bahyrycz]{Anna Bahyrycz}
\email[A. Bahyrycz]{bah@up.krakow.pl}
\address{Department of Mathematics, Pedagogical University of Cracow, Cracow, Poland}

\author[Zs. Páles]{Zsolt P\'ales}
\email[Zs. Páles]{pales@science.unideb.hu}
\address{Institute of Mathematics, University of Debrecen, Hungary}

\author[M. Piszczek]{Magdalena Piszczek}
\email[M. Piszczek]{magdap@up.krakow.pl}
\address{Department of Mathematics, Pedagogical University of Cracow, Cracow, Poland}

\title{Asymptotic stability of the Cauchy and Jensen functional equations}

\thanks{This research of the second author was supported by the 
Hungarian Scientific Research Fund (OTKA) Grant K111651.}


\date{\today}

\begin{document}
\begin{abstract}
The aim of this note is to investigate the asymptotic stability behaviour of the Cauchy and Jensen functional 
equations. Our main results show that if these equations hold for large arguments with small error, then they are 
also valid everywhere with a new error term which is a constant multiple of the original error term. As consequences, 
we also obtain results of hyperstability character for these two functional equations.  
\end{abstract}

\maketitle

\section{Introduction}

The stability theory of the classical Cauchy and Jensen functional equations has attracted a~lot of attention in the 
last decades. Historically, this theory was stimulated by a question of S.\ M.\ Ulam \cite{Ula60}, though the first 
result of this kind is due to Gy.\ P\'olya and G.\ Szeg\H{o} \cite{PolSze25}. The majority of the results concentrates 
on proving that if a function satisfies a given functional equation approximately, then it approximates a function which 
is the exact solution of the given functional equation. To demonstrate this phenomenon, we recall D. H. Hyers' 
classical theorem (\cite{Hye41a}) which states that if, $X$ is a linear space and $Y$ is a Banach space and, for some 
$\varepsilon\geq0$, a~function $f\colon X\to Y$ is $\varepsilon$-additive, i.e., 
\begin{equation*}
 \|f(x)+f(y)-f(x+y)\|\leq\varepsilon \qquad(x,y\in X),
\end{equation*}
then there exists an additive function $g\colon X\to Y$ such that $f$ is $\varepsilon$-close to $g$, i.e.,
\begin{equation*}
 \|f(x)-g(x)\|\leq\varepsilon  \qquad (x\in X).
\end{equation*}
Motivated by this result several contributions have been obtained since then (\cite{Aok50,Bou51a,Ras78,Mal08}). 
For the interested reader, we recommend the following books and surveys: Ger \cite{Ger94b}, Forti \cite{For95}, 
Czerwik \cite{Cze02}, Hyers, Isac and Rassias \cite{HyeIsaRas98}. There are at least four significantly different 
approaches to obtain stability theorems. The first one, the so-called direct method, or iterative method was already 
invented by Hyers in \cite{Hye41a} (see also \cite{BarVol02}, \cite{For04}, \cite{PalVolLuc98}, \cite{Sza13a}). Results 
using the technique of invariant means (over amenable semigroups) were first proved by L.\ Székelyhidi \cite{Sze87c} 
(see also \cite{Bad93}, \cite{BadGerPal03}, \cite{BadPalSze99}). The third method is to use variants of the 
Hahn--Banach separation theorem or sandwich theorems, or more generally selections theorems, see \cite{Pal98b}, 
\cite{NikPalWas99}. The most recently discovered method is to use fixed point theorems suggested by L.\ C\u{a}dariu and 
V.\ Radu \cite{CadRad02}, \cite{CadRad03} (see also \cite{BrzChuPal11}).

In the present note, we do not apply any of the general patterns to obtain stability theorems. Instead, we assume that 
the stability of the functional equation holds only for large values of the variables and we deduce that the stability 
holds on the entire domain. Such results explicitely or implicitely have already been obtained by S.-M.\ Jung 
\cite{Jun98}, by L.\ Losonczi \cite{Los96}, and by F.\ Skof \cite{Sko83c}. To formulate the main results of this paper, 
the most convenient structures for the domain and the codomain of the given functions are metric abelian groups. A 
triple $(X,+, d)$ is called a~metric abelian group if $(X, +)$ is an abelian group which is equipped with a translation 
invariant metric $d$. In this case we define $\|x\|_d := d(x, 0)$ and we call $\| \cdot \|_d$ the norm induced by the 
metric $d$. Then $\| \cdot \|_d$ is an even subadditive function on $X$. The standard example for a metric abelian group 
is the additive group of a normed space. We note that, in general, the norm on a metric abelian group is not 
necessarily positively homogeneous. By the subadditivity of the norm, the inequality $\|2x\|_d\leq2\|x\|_d$ is always 
valid, however the equality may fail.

\section{Results on the stability}

\begin{thm}	\label{t1}
Let $(X,+, d)$ and $(Y,+,\rho)$ be metric abelian groups such that $X$ is
unbounded by~$d$. Let $\varepsilon \geq 0$ and assume that $f \colon  X \to Y$ possesses the
following asymptotic stability property
\begin{equation}\label{e1t1}
\limsup_{\min(\|x\|_d, \|y\|_d)\to \infty} \|f(x+y)-f(x)-f(y)\|_{\rho}\leq \varepsilon,
\end{equation}
 then
\begin{equation}\label{e2t1}
\|f(x+y)-f(x)-f(y)\|_{\rho}\leq 5 \varepsilon \qquad \mbox{for all } x,y \in X.
\end{equation}
Furthermore, provided that \eqref{e1t1} holds, the constant $5\varepsilon$ is the smallest 
possible in \eqref{e2t1}.
\end{thm}

\begin{proof}
Let $\eta >\varepsilon$ be arbitrary. Then, by the asymptotic stability property \eqref{e1t1} of $f$, 
there exists $r>0$ such that, for all $x,y\in X$ with $\|x\|_d\geq r$, $\|y\|_d\geq r$,
\begin{equation}\label{eta}
 \|f(x+y)-f(x)-f(y)\|_{\rho}< \eta.
\end{equation}
Let $x,y\in X$ be fixed. Using the unboundedness of $X$, choose $u\in X$ first and then $v\in X$ such that
\[
   \|u\|_d \geq r+\|x\|_d \qquad\mbox{and}\qquad
   \|v\|_d \geq r+\|x\|_d+ \|y\|_d + \|u\|_d.
\]
Then, one can easily see that
\[
 \min\big(\|u\|_d,\|v\|_d,\|u+v\|_d,\|x-u\|_d,\|y-v\|_d,\|x+y-u-v\|_d\big)\geq r.
\]
Applying \eqref{eta} five times, we get
\begin{align*}
\|f(x-u)+f(u)-f(x)\|_{\rho}< &\eta, \\
\|f(y-v)+f(v)-f(y)\|_{\rho}< &\eta, \\
\|f(x+y-u-v)-f(x-u)-f(y-v)\|_{\rho}< &\eta, \\
\|f(u+v)-f(u)-f(v)\|_{\rho}< &\eta, \\
\|f(x+y)-f(x+y-u-v)-f(u+v)\|_{\rho}< &\eta.
\end{align*}
By, the triangle inequality, we obtain
\begin{align*}
\|f(x+y)-f(x)-f(y)\|_{\rho}
&\leq
\|f(x-u)+f(u)-f(x)\|_{\rho}\\
&\quad \mbox{}+\|f(y-v)+f(v)-f(y)\|_{\rho} \\
&\quad \mbox{}+\|f(x+y-u-v)-f(x-u)-f(y-v)\|_{\rho} \\
&\quad \mbox{}+\|f(u+v)-f(u)-f(v)\|_{\rho} \\
&\quad \mbox{}+\|f(x+y)-f(x+y-u-v)-f(u+v)\|_{\rho}\\
&< 5\eta.
\end{align*}
Taking the limit $\eta \to \varepsilon$ in the last inequality, we arrive at \eqref{e2t1}, which was to be proved.

To show that $5\varepsilon$ is the best constant in \eqref{e2t1}, let $x_0\in X\setminus\{0\}$ 
be an arbitrary element and define a function $f\colon X\to \R$ as follows:
\[
f(x)=\begin{cases}
\varepsilon, & \mbox{for } x\in X\setminus\{x_0\},\\
3\varepsilon, & \mbox{for } x=x_0.
\end{cases}
\]
Then, for all $x,y\in X$ with $\|x\|_d, \|y\|_d > \|x_0\|_d$, we have $f(x)=f(y)=\varepsilon$ and 
$f(x+y)\in\{\varepsilon,3\varepsilon\}$. Hence, for large values of $x$ and $y$,
\[
|f(x+y)-f(x)-f(y)|=\varepsilon,
\]
which implies the validity of \eqref{e1t1}. On the other hand, using that $2x_0\neq x_0$, we obtain
\[
\sup_{x,y\in X}|f(x+y)-f(x)-f(y)|\geq|f(2x_0)-2f(x_0)|=5\varepsilon,
\]
proving that the constant on the right hand side of \eqref{e2t1} cannot be smaller than $5\varepsilon$.
\end{proof}

Our next result is related to the asymptotic stability of the Jensen functional equation.

\begin{thm}	    \label{t2}
Let $(X,+, d)$ and $(Y,+, \rho)$ be metric abelian groups such that $X$ is
uniquely 2-divisible and unbounded by $d$. Let $\varepsilon \geq 0$ and assume that
$f \colon X \to  Y$ possesses the following asymptotic stability property
\begin{equation}\label{e1t2}
\limsup_{\min(\|x\|_d, \|y\|_d)\to \infty} \Big\|2f\Big(\frac{x+y}{2}\Big)-f(x)-f(y)\Big\|_{\rho}\leq \varepsilon,
\end{equation}
 then
\begin{equation}\label{e2t2}
\Big\|4f\Big(\frac{x+y}{2}\Big)-2f(x)-2f(y)\Big\|_{\rho}\leq 4 \varepsilon \qquad \mbox{for all } x,y \in X.
\end{equation}
Furthermore, provided that \eqref{e1t2} holds, the constant $4\varepsilon$ is the smallest 
possible in \eqref{e2t2}.
\end{thm}

\begin{proof}
Let $\eta >\varepsilon$ be arbitrary. Then, applying \eqref{e1t2}, there exists $r>0$ such that, 
for all $x,y\in X$ with $\|x\|_d\geq r$, $\|y\|_d\geq r$,
\[
\Big\|2f\Big(\frac{x+y}{2}\Big)-f(x)-f(y)\Big\|_{\rho}< \eta.
\]
Let us fix $x,y\in X$ and choose $u\in X$ such that
\[
\|u\|_d\geq r+\|x\|_d+\|y\|_d.
\]
Then, as $\min\big(\|x+u\|_d, \|y+u\|_d, \|x-u\|_d, \|y-u\|_d\big)\geq r$,
we get the following four inequalities
\begin{align*}
 \Big\|2f\Big(\frac{x+y}{2}\Big)-f(x+u)-f(y-u)\Big\|_{\rho}< &\eta, \\
\Big\|2f\Big(\frac{x+y}{2}\Big)-f(x-u)-f(y+u)\Big\|_{\rho}< &\eta, \\
\|f(x+u)+f(x-u)-2f(x)\|_{\rho}< &\eta, \\
\|f(y+u)+f(y-u)-2f(y)\|_{\rho}< &\eta.
\end{align*}
Thus, by the triangle inequality, we obtain
\begin{align*}
\lefteqn{\Big\|4f\Big(\frac{x+y}{2}\Big)-2f(x)-2f(y)\Big\|_{\rho}}\hspace{2em}\\
&\leq \Big\|2f\Big(\frac{x+y}{2}\Big)-f(x+u)-f(y-u)\Big\|_{\rho} +
\Big\|2f\Big(\frac{x+y}{2}\Big)-f(x-u)-f(y+u)\Big\|_{\rho}\\
&\quad{}+\|f(x+u)+f(x-u)-2f(x)\|_{\rho}+\|f(y+u)+f(y-u)-2f(y)\|_{\rho}\\
&< 4\eta.
\end{align*}
Upon taking the limit $\eta \to \varepsilon$, the desired inequality \eqref{e2t2} results.

To verify that $4\varepsilon$ is the best constant in \eqref{e2t2}, 
it suffices to consider the function $f \colon X\to \R$ defined by
\[
f(x)=\begin{cases}
0, & \mbox{for } x\in X\setminus\{x_0,-x_0, 0\},\\
-\frac{\varepsilon}{2}, & \mbox{for } x\in \{x_0,-x_0\},\\
+\frac{\varepsilon}{2}, & \mbox{for } x=0,
\end{cases}
\]
where $x_0\in X\setminus\{0\}$ is an arbitrary element. 
Then, for all $\|x\|_d, \|y\|_d > \|x_0\|_d$,
\[
\Big|2f\Big(\frac{x+y}{2}\Big)-f(x)-f(y)\Big|\in \{0,\varepsilon\},
\]
wich shows that \eqref{e1t2} is satisfied. On the other hand,
\[
\sup_{x,y\in X}\Big|4f\Big(\frac{x+y}{2}\Big)-2f(x)-2f(y)\Big|
\geq |4f(0)-2f(x_0)-2f(-x_0)|=4\varepsilon,
\]
proving that the constant on the right hand side of \eqref{e2t2} cannot be smaller than $4\varepsilon$.
\end{proof}

Taking $\varepsilon=0$ in Theorems \ref{t1} and \ref{t2}, we can directly obtain the following corollaries.

\begin{cor}
Let $(X,+, d)$ and $(Y,+,\rho)$ be metric abelian groups such that $X$ is
unbounded by $d$. If $f \colon  X \to Y$  satisfies
\[
\limsup_{\min(\|x\|_d, \|y\|_d)\to \infty} \|f(x+y)-f(x)-f(y)\|_{\rho}=0,
\]
then
 \[
 f(x+y)=f(x)+f(y),  \qquad x,y \in X.
 \]
\end{cor}

\begin{cor}
Let $(X,+, d)$ and $(Y,+, \rho)$ be metric abelian groups such that $X$ is
uniquely 2-divisible and unbounded by $d$.
If $f\colon X\to Y$ satisfies
\[
\limsup_{\min(\|x\|_d, \|y\|_d)\to \infty} \Big\|2f\Big(\frac{x+y}{2}\Big)-f(x)-f(y)\Big\|_{\rho}=0,
\]
 then
\[
 4f\Big(\frac{x+y}{2}\Big)=2f(x)+2f(y),  \qquad x,y \in X.
\]
\end{cor}

\section{Results on the hyperstability}

\begin{thm}		\label{t5}	
Let $(X,+, d)$ and $(Y,+,\rho)$ be metric abelian groups such that $2X$ is 
unbounded by $d$. Let $\varphi\colon\R_+:=[0,\infty)\to\R$ such that $\lim_{t\to\infty}\varphi(t)=\infty$ and
let $f\colon X\to Y$ satisfy
\begin{equation} \label{hipC}
\limsup_{\min(\|x\|_d, \|y\|_d)\to \infty} \varphi(\|x-y\|_d)\cdot\|f(x+y)-f(x)-f(y)\|_{\rho}<\infty,
\end{equation}
then
\[
f(x+y)=f(x)+f(y)\qquad  \mbox{for all } x,y \in X.
\]
\end{thm}

\begin{proof}
According to \eqref{hipC}, there exist constants $r>0$ and $K>0$ such that, for $x,y\in X$ with $\|x\|_d\geq r$ and
$\|y\|_d\geq r$,
\[
  \varphi(\|x-y\|_d)\cdot\|f(x+y)-f(x)-f(y)\|_{\rho}<K.
\]
Let $\varepsilon>0$ be arbitrary and choose $R>0$ such that, for all $t\geq R$,
\[
  \varphi(t)\geq \frac{5K}{\varepsilon}.
\]
Then, for $x,y\in X$ with $\|x\|_d\geq r$, $\|y\|_d\geq r$ and $\|x-y\|_d\geq R$, we have
\begin{equation} \label{hipD}
  \|f(x+y)-f(x)-f(y)\|_{\rho}<\frac{\varepsilon}{5}. 
\end{equation}
Let $x,y \in X$ be fixed. In view of the unboundedness of $2X$, we can choose $u\in X$ such that
\[
\|2u\|_d\geq 2r+R+2\|x\|_d
\]
and then we can choose $v\in X$ satysfying
\[
\|2v\|_d\geq 2r+2R+2\|x\|_d+2\|y\|_d+2\|u\|_d.
\]
Then, using the trivial inequality $\|2z\|_d\leq2\|z\|_d$ (which is the consequence of the subadditivity of 
$\|\cdot\|_d$), we easily get
\begin{align*}
  \min\big(\|u\|_d,\|v\|_d,\|u+v\|_d,\|x-u\|_d,\|y-v\|_d,\|x+y-u-v\|_d\big) & \geq r, \\
  \min\big(\|x-2u\|_d,\|y-2v\|_d,\|x-y-u+v\|_d,\|u-v\|_d,\|x+y-2(u+v)\|_d\big)&\geq R.
\end{align*}
Thus, using inequality \eqref{hipD} on the domain indicated, we obtain the following five inequalities
\begin{align*}
\|f(x-u)+f(u)-f(x)\|_{\rho}< &\frac{\varepsilon}{5}, \\
\|f(y-v)+f(v)-f(y)\|_{\rho}< &\frac{\varepsilon}{5}, \\
\|f(x+y-u-v)-f(x-u)-f(y-v)\|_{\rho}< &\frac{\varepsilon}{5}, \\
\|f(u+v)-f(u)-f(v)\|_{\rho}< &\frac{\varepsilon}{5}, \\
\|f(x+y)-f(x+y-u-v)-f(u+v)\|_{\rho}< &\frac{\varepsilon}{5}.
\end{align*}
Adding up these inequalities side by side and using the triangle inequality (as in the proof of Theorem \ref{t1}),
we get
\[
  \|f(x+y)-f(x)-f(y)\|_d\leq \varepsilon.
\]
Since $\varepsilon>0$ was arbitrary, this inequality implies that $f$ is additive.
\end{proof}

We cannot omit the assumption that $2X$ is unbounded by $d$  in Theorem \ref{t5}. 
Indeed, if $\sup_{x\in X}\|2x\|_d=K<\infty$, then the function $f\colon X\to X$ defined by
\[
f(x)=\begin{cases}
x, & \quad  x\in X\setminus\{0\},\\
a, & \quad  x=0,
\end{cases}
\]
where $a\neq 0$, for $x,y\in X\setminus\{0\}$ satisfies
\[
  \|x-y\|_d\cdot\|f(x+y)-f(x)-f(y)\|_{d}
  =\begin{cases}
  0 & \mbox{if } x+y\neq0,\\
  \|2x\|_d\cdot \|a\|_d & \mbox{if } x+y=0.\\
  \end{cases}
\]
Hence
\[
\limsup_{\min(\|x\|_d, \|y\|_d)\to \infty} \|x-y\|_d\cdot\|f(x+y)-f(x)-f(y)\|_{d}\leq K\cdot\|a\|_d<\infty,
\]
which means that \eqref{hipC} is fulfilled with $\varphi(t):=t$. On the other hand, $f$ is not additive.

\begin{thm}	
Let $(X,+, d)$ and $(Y,+, \rho)$ be metric abelian groups such that $X$ is uniquely 2-divisible and $2X$ is unbounded 
by $d$. Let $\varphi\colon\R_+\to\R$ such that $\lim_{t\to\infty}\varphi(t)=\infty$ and let $f\colon X\to Y$ satisfy 
\begin{equation} \label{je}
\limsup_{\min(\|x\|_d, \|y\|_d)\to \infty}  
\varphi(\|x-y\|_d)\cdot\Big\|2f\Big(\frac{x+y}{2}\Big)-f(x)-f(y)\Big\|_{\rho} < \infty,
\end{equation}
then
\begin{equation} \label{rj}
4f\Big(\frac{x+y}{2}\Big)=2f(x)+2f(y) \qquad \mbox{for all } x,y \in X.
\end{equation}
\end{thm}

\begin{proof}
Similarly as in the proof of Theorem \ref{t5}, it follows from \eqref{je} that there exist constants $r>0$ and $K>0$ 
such that, for $x,y\in X$ with $\|x\|_d\geq r$ and $\|y\|_d\geq r$,
\[
  \varphi(\|x-y\|_d)\cdot\Big\|2f\Big(\frac{x+y}{2}\Big)-f(x)-f(y)\Big\|_{\rho}<K.
\]
Let $\varepsilon>0$ be arbitrary and choose $R>0$ such that, for all $t\geq R$,
\[
  \varphi(t)\geq \frac{4K}{\varepsilon}.
\]
Then, for $x,y\in X$ with $\|x\|_d\geq r$, $\|y\|_d\geq r$ and $\|x-y\|_d\geq R$, we have
\begin{equation} \label{ji}
  \Big\|2f\Big(\frac{x+y}{2}\Big)-f(x)-f(y)\Big\|_{\rho}<\frac{\varepsilon}{4}.
\end{equation}
Let $x,y\in X$ be fixed and, using the unboundedness of $2X$, choose $u\in X$ such that
\[
  \|2u\|_d\geq 2r+R+2\|x\|_d+2\|y\|_d.
\]
Then, by the triangle inequality, it follows that
\begin{align*}
\min\big(\|x+u\|_d, \|y+u\|_d, \|x-u\|_d, \|y-u\|_d\big)&\geq r,\\
\min\big(\|2u\|_d,  \|x-y+2u\|_d, \|x-y-2u\|_d&\geq R.
\end{align*}
Applying these conditions for the appropriate choice of variables, \eqref{ji} implies the 
following four inequalities:
\begin{align*}
\Big\|2f\Big(\frac{x+y}{2}\Big)-f(x+u)-f(y-u)\Big\|_{\rho}< &\frac{\varepsilon}{4},\\
\Big\|2f\Big(\frac{x+y}{2}\Big)-f(x-u)-f(y+u)\Big\|_{\rho}< &\frac{\varepsilon}{4},\\
\|f(x+u)+f(x-u)-2f(x)\|_{\rho}< &\frac{\varepsilon}{4},\\
\|f(y+u)+f(y-u)-2f(y)\|_{\rho}< &\frac{\varepsilon}{4}.
\end{align*}
Thus
\[
 \Big\|4f\Big(\frac{x+y}{2}\Big)-2f(x)-2f(y)\Big\|_{\rho}< \varepsilon
\]
and as $\varepsilon>0$ was arbitrary, \eqref{rj} holds.
\end{proof}

We complete this paper by providing an example of an unbounded metric abelian group $X$ such that $2X$ is bounded.

\begin{ex}
Let
\[
X:=\{(a_1,a_2, \ldots):\  a_i\in \{0,1\}, \ i\in \N,\ \exists_{n_0\in\N} \forall_{i\geq n_0}\, a_i=0\}.
\]
For $(a_1,a_2, \ldots), (b_1,b_2, \ldots)\in X$, we define
\[
(a_1,a_2, \ldots) + (b_1,b_2, \ldots)=((a_1+_2 b_1, a_2+_2 b_2, \ldots),
\]
(where the operation $+_2$ is the additon modulo 2 on $\{0,1\}$) and
\[
d((a_1,a_2, \ldots), (b_1,b_2, \ldots)) =\sum_{i=1}^{\infty}\frac{a_i+_2 b_i}{i}.
\]
Then $(X,+,d)$ is an unbounded metric abelian group but $2X=\{(0,0,\ldots)\}$, hence $2X$ is trivially 
bounded.
\end{ex}


\end{document}